\documentclass{amsart}
\usepackage[mathscr]{eucal}
\usepackage{xypic}
\usepackage{amsfonts}   
\usepackage{amsmath}
\usepackage{amsthm}
\usepackage{amssymb}
\usepackage{latexsym}
\newfont{\msam}{msam10}

\newtheorem{theorem}[]{Theorem}
\newtheorem{proposition}[]{Proposition}
\newtheorem{corollary}[]{Corollary}
\newtheorem{lemma}[]{Lemma}
\theoremstyle{definition}
\newtheorem{definition}[]{Definition}
\newtheorem{remark}[]{Remark}

\let\nc\newcommand

\nc{\la}{\label}

\def\bthm{\begin{theorem}}
\def\ethm{\end{theorem}}
\def\blemma{\begin{lemma}}
\def\elemma{\end{lemma}}
\def\bproof{\begin{proof}}
\def\eproof{\end{proof}}
\def\bprop{\begin{proposition}}
\def\eprop{\end{proposition}}
\def\bcor{\begin{corollary}}
\def\ecor{\end{corollary}}

\def\D{\mathcal{D}}

\def\tA{\tilde{A}}

\def\c{\mathbb{C}}

\nc{\Hom}{{\rm{Hom}}}
\nc{\chara}{{\rm{char}}}
\nc{\Ext}{{\rm{Ext}}}
\nc{\HOM}{\underline{\rm{Hom}}}
\nc{\EXT}{\underline{\rm{Ext}}}
\nc{\TOR}{\underline{\rm{Tor}}}
\nc{\End}{{\rm{End}}}
\nc{\GL}{{\rm{GL}}}
\nc{\SL}{{\rm{SL}}}
\nc{\Rep}{{\rm{Rep}}}
\nc{\ad}{{\rm{ad}}}
\nc{\dlim}{\varinjlim}

\newcommand{\Frac}{{\rm{Frac}}}

\newcommand{\Spec}{{\rm{Spec}}}

\newcommand{\Aut}{{\rm{Aut}_{\c}}}

\newcommand{\id}{{\rm{Id}}}

\title{Noncommutative Noether's problem for  complex reflection groups}

\author{Farkhod Eshmatov}
\address{Department of Mathematics, University of Sichuan, Chengdu, China}
\email{olimjon55@hotmail.com}
\author{Vyacheslav Futorny}
\address{Instituto de Matem\'atica e Estat\'\i stica, Universidade de S˜\~ ao Paulo\, S\~ao Paulo, Brasil}
\email{futorny@ime.usp.br}
\author{Sergiy Ovsienko}
\author{Joao Fernando Schwarz}
\address{Instituto de Matem\'atica e Estat\'\i stica, Universidade de S˜\~ ao Paulo\, S\~ao Paulo, Brasil}
\email{jfschwarz.0791@gmail.com}

\date{}                                           

\begin{document}

\maketitle

\begin{abstract}
We solve some noncommutative analogue of the
Noether's problem for the reflection groups by showing that the skew  field of fractions of the invariant subalgebra  of the Weyl algebra under the action of any   finite complex reflection group is a Weyl field, that is isomorphic to the skew field of fractions of some Weyl algebra. We also extend this result to the invariants of the ring of differential 
   operators on any finite dimensional torus. The results are applied to obtain  analogs of the Gelfand-Kirillov Conjecture for Cherednik algebras and Galois algebras. 
\end{abstract}

\section{Introduction}
Let $k$ be a field of characteristic $0$.  Let $G$ be a finite subgroup of $\GL_n(k)$
acting naturally by linear automorphisms on $S:=k[x_1, \ldots, x_n]$, and hence on the 
field of fractions $F:=\Frac(S)= k(x_1, \ldots, x_n)$. It is easy to show  $F^G=\Frac(S^G)$.
Then a well-known result due to E.Artin claims that the transcendence degree of
$F^G$ over $k$ is equal to $n$.  Now one can ask the following natural question:

$\textit{Noether's  problem. }$  Is $F^G$  a purely transcendental 
extension of $k$ or equivalently is $F^G$  isomorphic to $F$?

This problem has been studied by many authors. Let us briefly recall some of the results.
For more detailed discussion see \cite{D}.

By a classical theorem of E.Fisher, the answer to Noether's problem is positive for all 
$n\ge 1$ when $G$ is abelian and $k$ algebraically closed. 
For a general  $k$ this is no longer true. One of the first counterexamples were
produced in \cite{Le} for $k=\mathbb{Q}$ and when $G$ is a cyclic group of order eight.
The answer is also positive for all $n\ge 1$ when $G$ is a  complex reflection group,
since it is a consequence of the  Chevalley-Shephard-Todd theorem: $ S^G$  is isomorphic to $S$.

In general, for $n=1$  the positive answer is a straightforward consequence  of the classical
 theorem of L\"{u}roth, while for $n=2$ it is  a simple consequence of Castelnuovo's theorem. 
For $n=3$  the positive answer  was proved by  Burnside using Miyata's theorem.

Now, we will discuss a noncommutative version of the Noether problem.
Let $A_n=A_n(k)$ be the $n$-th Weyl algebra with usual generators $x_1, \ldots, x_n$ and $\partial_1, \ldots, \partial_n$, and let $F_n$ be its skew field of  fractions. 
Then the action of $G$ on $S$ naturally extends to $A_n$ and to $F_n$.
The following question was originally posed by J.Alev and F.Dumas 
(see \cite[Section $1.2.2$]{AD})

$ \textit{Noether's problem for} \,  \, A_n $: 
Is  $F_n^G$  isomorphic to $F_n \, ?$ \\
One of the motivations to study this problem comes from \cite[Theorem $5$]{L}, where  it has been 
shown that the the subalgebra of the invariant differential operators on the affine space under the action 
of a finite unitary reflection group $G$ is not isomorphic to the whole algebra of differential operators, that is  $A_n^G$  is not isomorphic to $A_n$.

Let $V$ be a finite dimensional vector space of dimension $n$ over $k$. By fixing a basis in $V$ one can identify $S(V^*)$  with $k[x_1, \ldots, x_n]$, where $x_1, \ldots, x_n$ is the dual basis in $V^*$. If $G$ be a finite subgroup of $GL(V)$ (more generally $V$ is a $G$-module) then it acts on $S(V^*)$ by linear automorphisms: $g.f(v)=f(g^{-1}v),$ $g \in G$, $f \in S(V^*)$, $v \in V$. This action can be naturally extended to the ring of differential operators $\D(k[x_1, \ldots, x_n])$ on $S(V^*)$. This induces a group of linear automorphisms  of the Weyl algebra $A_n$. 
The following was proved in  \cite{AD}

\begin{theorem}
$(a)$ Let $V$ be a  representation of $G$  which is a direct sum of $n$ representations of 
dimension one. Then $F_n^G \cong F_n$.\\
$(b)$ For any $2$-dimensional representation of $G$, we have $F_2^G \cong F_2$.
\end{theorem}
It follows from part $(a)$ that we have the positive answer to Noether's problem 
for $A_n$ for all $n \ge 1$ when $G$ is abelian and $k$ is
algebraically closed.    We should also point out that in \cite{AD}, this problem 
was discussed for the case when $G$ is not necessarily a finite group. 
In this case the above question should 
be slightly modified.

From now on we assume that all algebras  and  varieties are defined over $\c$.

To our knowledge, the only other case for which  Noether's problem for $A_n$ has been
considered is when $G \cong S_n$, a symmetric group of degree $n$, acting on $S(V^*)$ 
by permuting variables $x_i$. 
In \cite{FMO}, it was shown among other things that $F_n^{S_n} \cong F_n$.

Our first main result in this paper is
\begin{theorem}
\la{main}
Let $V$ be an $n$-dimensional vector space over $\c$ and
 $W$ be a finite complex reflection subgroup of $\GL(V)$. Then $F_n^W \cong F_n$.
\end{theorem}
Perhaps, this result is known to specialists but we could not find any proof. We are aware of an unpublished manuscript by I.Gordon where a similar statement is claimed without a proof. 

Next we extend this technique to the study of the skew field of fractions of the invariants for classical reflection groups in the case of any finite dimensional torus.

Our second main result  is
\begin{theorem}
\la{main-2}
Let $X=T^n$ be an $n$-dimensional torus, $\D(X)$ the ring of differential operators on $X$, $F(\D(X))$ the skew filed of fractions of $\D(X)$,  $W$ a classical complex  reflection group.  Then  there exists a natural action of $W$ on  
 $\D(X)$ which extends to  $F(\D(X))$ and  
 $F(\D(X))^W \cong F_n$.
\end{theorem}

As one of the applications of Theorem~\ref{main}, we will show that  an analogue  of the Gelfand-Kirillov conjecture for
Lie algebras holds for spherical subalgebras of rational Cherednik algebras $H_k:=H_k(W)$ associated
to $W$. We also discuss the Gelfand-Kirillov conjecture for a class of linear Galois algebras which includes the universal enveloping algebras of $gl_n$ and $sl_n$. 

\bigskip

\noindent{\bf Acknowledgements.} V.F. would like to thank the Mittag-Leffler Institute for its hospitality during his stay where part of this work was done.
F.E. is supported in part by 
Fapesp  ( 	
2013/22068-6), V.F. is
supported in part by  CNPq  (301320/2013-6) and by 
Fapesp  (2014/09310-5).

\section{Invariant differential operators}

Let $B$ be a commutative algebra. The ring of differential operators  $\D(B)$ is
defined to be $\D(B)=\cup_{n=0}^{\infty}\D(B)_n$, where $\D(B)_0=B$ and
$$ \D(B)_n \, = \, \{ \, d \in \End_{\c} (B) \, :\, d\, b - b\, d \in \D(B)_{n-1}\, \mbox{ for all }\, b \in B \} \, .$$

Let $B$ be a reduced, finitely generated algebra and let $G$ be a group
acting on $B$ by algebra automorphisms. Then $G$ acts on $\D(B)$,
via  $ (g \ast \partial) \cdot f = g \circ \partial \circ g^{-1} \cdot f$.
By \cite[Theorem 5]{L} restricting differential operators gives an 
injective homomorphism $ \D(B)^{G} \to \D(B^G)$. It is interesting
to know when this map is an isomorphism. The following
is a special case of \cite[Theorem 3.7]{CH}

\begin{theorem}
\la{cor1}
Let  $X$ be a normal, irreducible, affine algebraic variety, and let  
$G$ be a finite group acting freely on $X$. Then $\D(X)^G \cong \D(X/G)$.
\end{theorem}

\section{Proof of Theorem \ref{main}.}
Let $V$ be a finite-dimensional vector space over $\c$. An element $s \in \GL(V)$
is a \textit{complex reflection} if it acts as identity on some hyperplane $H_s$ in $V$.
A finite subgroup $W$ of $\GL(V)$ is called a  complex reflection group if
it is generated by its complex reflections.  Let $(\cdot ,\cdot)$ be a positive definite Hermitian form on $V$, which is 
invariant under the  action of $W$. We may assume that 
$(\cdot, \cdot)$ is antilinear in the first argument and linear on its second argument:
if $x\in V $, we write $x^*$ for the linear form $V \to \c, \, v \mapsto (x,v)$.

Let $\mathcal{A}=\{H_s \}$ denote the set of reflection hyperplanes of $W$,
corresponding to $s\in W$. The group $W$ acts on $\mathcal{A}$ by permutations.
If $H \in \mathcal{A}$, the (pointwise) stabilizer of $H$ in $W$ is a cyclic subgroup 
$W_H$ of order $n_H$.
Let $\alpha_{H}$ be a linear form for which $H$ is the zero set. It is defined up to a constant.
We set 
$$ \delta \, := \, \prod_{H\in \mathcal{A}} \alpha_H   \quad , \quad   J\, := \, \prod_{H\in \mathcal{A}} \alpha_{H}^{n_H-1}\, .  $$ 
It is easy to show that $w.J = \det(w) J$ for any $w\in W$ (see   \cite[Exercise $4.3.5$]{Sp}). 
Let $N$ be the order of $W$. Then $\triangle :=J^N$ is an invariant polynomial.

We fix a basis $\{v_1,...,v_n\}$ of $V$ and  let $\{x_1,...,x_n\}$ be the corresponding dual basis of $V^*$.
Then $S:=\c[V]= \c[x_1,...,x_n]$. Let $S_\delta$ be the localization of $S$ by $\{1,\delta, \delta^2, ... \}$.
Then $S_\delta=S_J=S_\triangle$ since they are just localizations by $\{ \alpha_H^k\}_{ k\ge 0, H \in \mathcal{A}}.$ 
In fact,  $S_\delta \cong \c[V^{\rm{reg}}]$, where $ V^{\rm{reg}}:=V \setminus \bigcup_{H\in \mathcal{A}} H$.

\begin{lemma}\label{lem-free}
The action of $W$ restricts to a free action on $V^{\rm reg}$ and on  $\c[V^{\rm reg}]$.

\end{lemma}

\begin{proof}
Assume that for some $w\in W$ and $v\in V^{\rm reg}$, $wv$ belongs to a hyperplane fixed by some reflection $s$. Then 
$w^{-1}sw$ belongs to the isotropy group of $v$ which is also a reflection group by the Steinberg's theorem.  Since $v\in V^{\rm reg}$, we conclude that $s=id$, 
which is a contradiction. Hence, $wv\in V^{\rm reg}$ and this action is clearly free.
\end{proof}

If we let $X:=\Spec(S_\triangle)$ and $Y:=\Spec((S_\triangle)^W)$, then both $X$ and $Y$ are normal, irreducible, 
affine algebraic varieties. Since the action of $W$ on $X$ is free, the dominant morphism 
$\phi : X \to Y$ is unramified in codimension $1$.
So we can use Theorem~\ref{cor1} to get 
\begin{equation}
\la{DiffIden}
 \D(S_\triangle)^W \, \cong \D((S_\triangle)^W) \,  .
 \end{equation}

\begin{proposition}\label{prop1}
$(i)$ If $A$ is a domain and $M$ is an Ore subset then $\Frac(A_M)\cong \Frac(A)$.\\
$(ii)$  $(S_\triangle)^W \cong (S^W)_\triangle$.\\
$(iii)$ $\D(S_\triangle)^W \cong (\D(S)^W)_\triangle$.\\
$(iv)$ $\Frac(A_n)^W\simeq \Frac(A_n^W)$.\\
\end{proposition}
\begin{proof}
$(i)$ This  statement is clear.\\
$(ii)$ Since $\triangle$ is an invariant polynomial then $f\in (S_\triangle)^W$ iff $\triangle^k f \in S^W$
for some $k\ge 0$ iff $f \in (S^W)_\triangle$.\\
$(iii)$ Note that $\D(S_M)\cong D(S)_M$ for a multiplicative set $M$, \cite[Theorem 15.1.25]{MR}.
If $d \in \D(S_\triangle)^W$ then $\triangle ^k d \in \D(S)^W$ for some $k\ge 0$.  Finally, (iv) follows from \cite{Fa}, Theorem 1, see also \cite{D}.
\end{proof}
 
Now  $(S_\triangle)^W \cong (S^W)_\triangle \cong S_{\triangle}$ , where the first identity holds by part $(ii)$
while the second one follows from the Chevalley-Shephard-Todd theorem.  Therefore, the right hand side of 
\eqref{DiffIden} is isomorphic to  $\D(S_\triangle)\cong \D(S)_\triangle $. Thus, using part $(ii)$ we have
$$ ( \D(S)^W)_\triangle \, \cong \, \D(S)_\triangle  \, .$$
Finally, taking the skew field of  fractions on both sides, we obtain
$$F_n^W \cong F_n \, .$$
\section{Gelfand-Kirillov conjecture for rational Cherednik algebras}
Let us first recall the definition of rational Cherednik algebras.
As before $W$ is  a finite complex reflection subgroup of $\GL(V)$
and $(\cdot, \cdot)$ is a $W$-invariant positive definite Hermitian form.
For $H\in \mathcal{A}$, let $v_H \in V $ be such that $\alpha_H=v_H^*$.
Next for each  $H$  from  $\mathcal{A}$, we set
$$ e_{H,i}\, := \, \frac{1}{n_H} \, \sum_{w\in W_H} \,(\det w )^{-i} w \, .$$ 
Since $W_H$ is a cyclic group of order $n_H$, this is a complete set of orthogonal idempotents in $\c W_H$. 
Now, for $H\in \mathcal{A}$, we fix  a sequence of non-negative
integers $k_H=\{k_{H,i}\}_{i=0}^{n_H-1}$ so that $k_H=k_{H'}$ if $H$ and
$H$ are on same orbit of $W$ on $\mathcal{A}$.
%

The \textit{rational Cherednik algebra} $H_k=H_k(W)$ is generated by elements
$x\in V^*,  \xi \in V $ and $w \in W$ subject to the following relations
\begin{eqnarray}
&& [x,x'] = 0 \, , \quad [ \xi,\xi']=0\, , \quad w\, x \, w^{-1} =w(x)\, , \quad w\, \xi  \, w^{-1} = w(\xi)\,  , \nonumber \\
&& [\xi,x] = \langle \xi , x \rangle + \sum_{H\in \mathcal{A}} \frac{\langle \alpha_H,  \xi \rangle\, 
\langle x, v_H \rangle  }{\langle \alpha_H , v_H \rangle } \sum_{i=0}^{n_H-1} n_H ( k_{H,i}-k_{H,i+1}) e_{H,i}\, . \nonumber 
\end{eqnarray}

Next, we introduce the \textit{spherical subalgebra} $U_k(W)$ of $H_k$: by definition, $U_k(W):=eH_ke$, where 
$e:=|W|^{-1} \sum_{w \in W} w$ is the symmetrizing idempotent in $\c W  \subset H_k$.
The skew field of fractions of the spherical subalgebra was studied by  Etingof and Ginzburg  \cite[Theorem $17.7^*$]{EG}. 
Combining this result with Theorem~\ref{main}, we get the following analogue of the Gelfand-Kirillov conjecture for rational Cherednik algebras:
\begin{theorem}
For a complex reflection group $W$ we have $\Frac(U_k (W)) \cong F_n$.
\end{theorem}
%

\section{Noether's problem for $n$-dimensional torus}

Let $X =\mathbb{T}_n= \mathrm{Spec} \,( k[x_1^{\pm 1}, \ldots, x_n^{\pm 1}])$ be 
the $n$-dimensional torus. Then  $\D(X)\simeq \Tilde{A_{n}}$ where
 $\tilde{A_{n}}$ is the localization of $A_{n}$ by the multiplicative set generated 
 by $\{x_i | i=1, \ldots, n\}$.  We consider the action of classical reflection groups on $\D(X)$.

For each $j=1,\ldots, n$ consider the  involutions $\tau_j^{\pm}$ on $\c[x_1^{\pm1}, \ldots, x_n^{\pm1}]$ such that
$$\tau_j^{\pm}(x_j) = \pm x_j^{-1}    \, \mbox{ and } \, \tau_j^{\pm}(x_i)=x_i  \, \mbox { for } \, i\neq j \, .$$ They
induce the involutions $\varepsilon_{n,j}^{\pm}$ on $\tilde{A_{n}}$ such that $\varepsilon_{n,j}^{\pm}(x_i)=\tau_j^{\pm}(x_i)$,
$$\varepsilon_{n,j}^{\pm}(\partial_j)=\mp x_j^{2}\partial_j  \, \mbox{ and } 
\, \varepsilon_{n,j}(\partial_i)=\partial_i \, \mbox{  if }  i\neq j \, ,$$
$i,j=1, \ldots, n$. 
We show it for $n=1$. 
Suppose we have a group action $G\times \mathbb{T}^1\rightarrow \mathbb{T}^1$ on one dimensional torus. Then the induced action on $\tilde{A}_{1}$ is given as
$$x \stackrel{g}\longmapsto x^{g}, \partial\stackrel{g}\longmapsto \partial^{g}= g\partial g^{-1}.$$
Hence, if  $\tau^{-}$ sends $x$ to $-x^{-1}$  then we have 

$$\varepsilon^{-}(\partial)(x^r)=\partial^{\tau^{-}}(x^r)= \tau^{-}\partial \tau^{-}(x^r)=(-1)^r \tau^{-}\partial(x^{-r})=(-1)^{r+1} r \tau^{-}(x^{-(r+1)})=$$
$=rx^{r+1},$ and  $\varepsilon^{-}(x)(x^r)=(-1)^r \tau^{-}(x^{-r+1})=-x^{r-1}.$\\

We obtain $\varepsilon^{-}(\partial)=x^{2}\partial$, $\varepsilon^{-}(x)=-x^{-1}$. This is easily generalized to $n$-dimensional torus $\mathbb{T}^n$ and  $\tau_i^{\pm}$, $i=1, \ldots, n$.

We will consider  the action of the reflection group of type $B_n$ ($n\geq 2$)  and the reflection group of type $D_n$ ($n\geq 4$) on $X$.  
 We recall, the group $B_n$ is the semi-direct product of the symmetric group $S_{n}$ and  
 $(\mathbb{Z}/2\mathbb{Z})^n$.  
There is a natural action of $B_n$
on $\D(X)=\tilde{A_{1}}^{\otimes n}$,  where $S_n$ acts by  permutations 
 and $(\mathbb{Z}/2\mathbb{Z})^n$  acts by $\varepsilon_{n,i}^{-}$, $i=1, \ldots, n$.




We have

\begin{proposition}
\label{proposition-action-in-odd-case}
%
(i)\,  The subalgebra of $B_n$-invariants of $\D(X)$  is a polynomial algebra in 
$$s_{i}=e_{i}(x_{1}-x_{1}^{-1},\dots,x_{n}-x_{n}^{-1}), \, i=1, \dots, n,$$ 
where $e_i$ is the $i$-th elementary symmetric polynomial.
In particular, $X/B_n$ is  $n$-dimensional affine space.

%
(ii) \, Let $Z\subset \mathbb{T}^{n}$ be the subvariety defined by the following equation 
$$\prod_{1\leq i\leq j\leq n} (x_{i}^{2}-\frac{1}{x_{j}^{2}})\prod_{1\leq i<j\leq n}(x_{i}^{2}-x_{j}^{2})=0$$ 
and $U=\mathbb{T}^{n}\setminus Z$. Then $U$ is an affine $B_n$-invariant subvariety of $X$ and the action of $B_n$ on $U$ is free. In particular, the projection $\pi:U\mapsto U/B_n$ is etale.
%
\end{proposition}

\begin{proof}
\label{proof-of-the-proposition-action-in-odd-case}
For $(i)$,  consider the lexicographical order on Laurent monomials.
Let $\pi=(k_{1},\dots,k_{n})$ be a sequence of integers with the property 
$k_{1}\geq k_{2}\geq\dots\geq k_{n}\geq 0$
and  $x_{1}^{k_{1}}\dots x_{n}^{k_{n}}$  the corresponding  monomial.   
Denote by $m_{\pi}=x_{1}^{k_{1}}\dots x_{n}^{k_{n}}+\dots$ a $B_n$-invariant 
polynomial with a minimal number of monomials. We will call $\pi$ the degree 
of $m_{\pi}$. The polynomials $m_{\pi}$ form a basis of the subalgebra of  
$B_n$-invariants. The leading monomial  $x_{1}^{k_{1}}\dots x_{n}^{k_{n}}$ 
of $m_{\pi}$  coincides with the leading monomial  of 
$M_{\pi}:=s_{1}^{k_{1}-k_{2}}\dots s_{n-1}^{k_{n-1}-k_{n}}s_{n}^{k_{n}}$. 
Then $m_{\pi}-M_{\pi}$ has a smaller leading monomial and we can proceed 
by induction on the degree. 

For $(ii)$ we denote 
$$\Delta=\prod_{1\leq i, j\leq n} \bigg(x_{i}^{2}-\frac{1}{x_{j}^{2}}\bigg)\, \prod_{1\leq i<j\leq n}(x_{i}^{2}-x_{j}^{2})\,
 \bigg(\frac{1}{x_{i}^{2}}-\frac{1}{x_{j}^{2}}\bigg)\, \prod_{i=1}^{ n}
 \bigg(x_{i}^{2}-\frac{1}{x_{i}^{2}}\bigg).$$ 
Then one can easily that $\Delta$ is $B_n$-invariant and $U=X\setminus V(\Delta)$, where $V(\Delta)$ is 
the algebraic subset of $\mathbb{T}^n$ corresponding to $\Delta$.
\end{proof}

 The  group $D_n$  is generated by $S_{n}$ and  $(\mathbb{Z}/2\mathbb{Z})^{n-1}$ which consists  
   of the transformations  $(\varepsilon_{1}^{d_{1}},\dots, \varepsilon_{n}^{d_{n}})\in 
   (\mathbb{Z}/2\mathbb{Z})^{n}$, $d_{i}=0,1,i=1,\dots, n$, such that $d_{1}+\dots+d_{n}$ is even. 
Consider now the action of
 $D_n$
on $\D(X)=\tilde{A_{1}}^{\otimes n}$  where $S_n$ acts by natural permutations and
$\varepsilon_i$ acts as
$$\id_{\tA_{1}}^{i-1}\otimes \varepsilon_{n,i}^{+}\otimes \varepsilon_{n,i+1}^{+}\otimes \id_{\tA_{1}}^{n-i-1}, i=1,\dots, n-1.
$$

\begin{proposition}
\label{proposition-action-in-even-case}
(i) The subalgebra of $D_n$-invariants of $\D(X)$ is generated by $$s_{i}=e_{i}(x_{1}+x_{1}^{-1},\dots,x_{n}+x_{n}^{-1}),\, i=1,\dots,n-1$$ and
$$ \Delta_{n}^{\pm}=\frac{1}{2}\bigg(\, \prod_{i=1}^{n}(x_{i}+\dfrac{1}{x_{i}})
\pm
\prod_{i=1}^{n}(x_{i}-\frac{1}{x_{i}})\,\bigg).$$
Moreover,  $\Delta_{n}^{-}\in \c[s_{1},\dots,s_{p-1},\Delta^{+}]_P$, where $P$ is  some polynomial in \\
$\c[s_{1},\dots,s_{n-1},\Delta_{n}^{+}]$.  In particular,  $X/D_n$ is isomorphic to a principal open subset of $n$-dimensional affine space.

(ii) Let $Z\subset \mathbb{T}^n$ be the variety defined by equation 
$\Delta=0$ 
and $U=X\setminus Z$. Then $U$ is an affine $D_n$-invariant subvariety of $X$ and the action of $D_n$ on $U$ is free. In particular, the projection $\pi:U \rightarrow U/D_n$ is etale.
\end{proposition}

\begin{proof}
\label{proof-of-the-proposition-action-in-even-case}
Proof of  $(i)$ is similar to the proof of  $(i)$ in Proposition \ref{proposition-action-in-odd-case}.  
Order the Laurent monomials lexicographically. 
Let $\pi=(k_{1}, \dots, k_{n})$ be a sequence of integers such that 
$k_{1}\geq k_{2}\geq\dots\geq |k_{n}|\geq 0$. Note that $k_{n}$ can be negative.
Set
$$\lambda_{\pi}=|\{g\in B_n\mid g\cdot (x_{1}^{k_{1}}\dots x_{n}^{k_{n}})
=x_{1}^{k_{1}}\dots x_{n}^{k_{n}}\}|, \, m_{\pi}=\lambda_{k}^{-1}
\sum_{g\in B_n}{}g\cdot (x_{1}^{k_{1}}\dots x_{n}^{k_{n}}).$$
 Then polynomials $m_{\pi}$ form a basis of the space of $D_n$-invariant Laurent polynomials. The leading monomial in $m_{\pi}$ is $x_{1}^{k_{1}}\dots x_{n}^{k_{n}}$ and
the same leading monomial has  the element 
$$M_{k}:=s_{1}^{k_{1}-k_{2}}\dots s_{n-1}^{k_{n-1}-k_{n}}(\Delta_{n}^{sign(k_{n})})^{|k_{n}|}.$$
Then $ m_{k}-M_{k}$ has a smaller leading monomial and we can proceed by induction. 
Next we show how to choose $P$. 
Note that both $s_n=\Delta^{+}_{n}+\Delta^{-}_{n}$ and $D=\Delta^{+}_{n}\Delta^{-}_{n}$ are $D_n$-invariant and $D$  can be expressed as a polynomial in $s_{1},\dots, s_{n}$. The  leading monomial in $D$ has the degree $(2,2,\dots,2,0)$ and, hence, $s_{n}$ can not enter in the expression for $D$ in the degree greater than $1$, as the degree of  the leading  monomial in $s_{n}$ is $(1,1,\dots,1,1)$.
 It is easy to see that the polynomial part of $D$ consists of the  squares and hence $D\not\in \c[s_{1},\dots,s_{n-1}]$ since it has the same leading monomial as $s_{n-1}^{2}$ and the second in lexicographical order monomial in $s_{n-1}^{2}$ has degree $(2,2,\dots,2,1,1)$. We conclude that
$$\Delta_{n}^{+}\Delta_{n}^{-}=p_{1}(s_{1},\dots,s_{n-1})
+s_{n}p_{0}(s_{1},\dots,s_{n-1}), \text{ i.e }
\Delta_{n}^{-}=\frac{\Delta_{n}^{+}p_{0}+p_{1}}{\Delta_{n}^{+}-p_{0}}.$$
 Set $P=\Delta_{n}^{+}-p_{0}.$  Then  $(i)$  follows.


To show $(ii)$, we take the same polynomial $\Delta$ as in the proof of 
Proposition \ref{proposition-action-in-odd-case}. Then 
$\Delta$ is $D_n$-invariant and $U=X\setminus V(\Delta)$.


\end{proof}

\subsection{Proof of Theorem 3}
 Let $X = \mathbb{T}^n = \mathrm{Spec} \,( k[x_1^{\pm 1}, \ldots, x_n^{\pm 1}])$, 
 $\Lambda = k[x_1,\ldots,x_n]$ and $\Gamma = \Lambda_f$, where $f=x_1\ldots x_n$. 
 Then $X = \mathrm{Spec} \, \Gamma$ is an affine, regular, normal irreducible variety. 
 Then the statement for the symmetric group $S_n$ is analogous to Theorem 2. 
 
We consider first the action of the group $B_n$. By Proposition \ref{proposition-action-in-odd-case}, 
the action of $B_n$ restricts to a free action on $U= \Spec \, \Gamma_{\Delta}$ which is an affine, 
irreducible, regular and normal variety.   By applying Theorem \ref{cor1}, we have  
$\D(U)^{B_n} \cong \D(U/B_n)$ and  $D(\Gamma_\Delta)^{B_n} \cong \D(\Gamma_\Delta^{B_n})$.
By Proposition \ref{prop1}, we may conclude
 $\D(\Gamma)^{B_n}_\Delta \cong \D(\Gamma^{B_n}_\Delta)$. 
 Since $\Gamma^{B_n} \cong \Lambda$ we have 
 $\D(\Gamma)^{B_n}_\Delta \cong \D(\Lambda_\Delta) \cong \D(\Lambda)_\Delta$.
  Forming the skew fields of fractions  we conclude  $\Frac \, \D(X)^{B_n} \cong  \Frac \, \D(X)$.

Consider now the case of the  group $D_n$. Repeating the same steps as above 
we have $\D(\Gamma)^{D_n}_\Delta \cong \D(\Gamma^{D_n}_\Delta)$. By
Proposition \ref{proposition-action-in-even-case}, 
 $\Gamma^{D_n} \simeq  \Lambda_P$ for some polynomial $P$. Therefore, $$\D(\Gamma)^{D_n}_{\Delta} \cong \D((\Gamma_P)_{\tilde{\Delta}})=\D(\Gamma_{P\Delta}) = \D(\Gamma)_{P\Delta}.$$ Forming the skew fields of fractions  we conclude $\Frac \, \D(X)^{D_n} \cong \Frac \, \D(X)\cong F_n$, which completes the proof of Theorem 3.

\section{Galois algebras}
Let $\Gamma$ be an
integral
 domain, $K$ the field of fractions of $\Gamma$,
$K\subset L$ is a finite Galois extension with the Galois
group $G$.  Let
$\mathcal M\subset \Aut L$ be a monoid on which $G$
acts by conjugations, 

Recall that
 an associative  $\c$-algebra $U$
containing $\Gamma$ is called a \emph{Galois $\Gamma$-algebra}{} with respect to $\Gamma$
if it is finitely generated $\Gamma$-subalgebra in $(L*\mathcal M)^{G}$
and $KU=(L*\mathcal M)^{G}, UK=(L*\mathcal M)^{G}$ \cite{FO}.

If $U$ is such algebra then  $S=\Gamma\setminus
\{0\}$ satisfies both left and right
Ore condition and the canonical embedding $U\hookrightarrow (L*\mathcal M)^{G}$ induces the isomorphisms of rings of fractions $[S^{-1}]U\simeq (L*\mathcal M)^{G}$, $
    U[S^{-1}]\simeq (L*\mathcal M)^{G}$.

The following is standard.

\begin{proposition}\label{corol-skew-field-invariants}
If $L*\mathcal M$ is an Ore domain, then $(L*\mathcal M)^G$ is an Ore
domain. If $\mathcal L$ is the skew field of fractions of $L*\mathcal M$,
then the skew field of fractions of $(L*\mathcal M)^G$ coincides
with $\mathcal L^G$, where  the action of $G$ on $\mathcal L$  is induced
by the action of $G$ on $L*\mathcal M$.
\end{proposition}

We immediately have

\begin{corollary}\label{corol-skew-field-Galois} Let $U$ be a Galois $\Gamma$-algebra and the skew group algebra $L*\mathcal M$ is the left and the right Ore domain with the skew field of fractions $\mathcal L$. Then $U$ is the left and right Ore domain and for its skew field of fractions
$\mathcal U$ holds $\mathcal U=\mathcal L^G$. In particular, all Galois
subalgebras  with respect to  $\Gamma$ in $(L*\mathcal M)^G$ have the
same skew field of fractions.

\end{corollary}

\begin{proof}
Due to
Proposition~\ref{corol-skew-field-invariants}, $U[S^{-1}]\simeq
(L*\mathcal M)^{G}$ is an Ore domain, $S=\Gamma\setminus \{0\}$.
Hence for any $u_{1},u_{2}\in U$ there exist
$v_{1},v_{2}\in U,s_{1},s_{2}\in S$, such that
$u_{1}v_{1}s_{1}^{-1}=u_{2}v_{2}s_{2}^{-1}$. Since $S$ is
commutative we get $u_{1}v_{1}s_{2}=u_{2}v_{2}s_{1}$. Other
 conditions  of Ore rings are proved analogously.
\end{proof}

Let  $V$ be a $\c$-vector space, $\dim_{\c}V=N$, $L$  the field of fractions of the symmetric algebra $S(V)$ and $G\subset \Aut(L)$ is the classical reflection group whose action on $L$ is induced  by its action by reflections on  $V$. Then $L=\c(t_{1},\dots, t_{N})$ and $K=L^G$ is a purely transcendental extension of $\c$ by the Chevalley-Shephard-Todd Theorem. If $\mathcal M\subset \Aut(L)$ is a subgroup such that $G$ normalizes $\mathcal M$ and $U$ is a Galois algebra in $(L*\mathcal M)^G$ with respect to a polynomial subalgebra $\Gamma$ such that the field of fractions   
 $F(\Gamma)$ is isomorphic to $ K$, then such $U$ will be called a \emph{linear Galois algebra}. In particular, $\mathcal K=(L*\mathcal M)^G$ is itself a linear 
Galois algebra with respect to $\Gamma$ if $L$, $G$, $\mathcal M$, $K$ and $\Gamma$ are as above. Other examples of linear Galois algebras are the universal enveloping algebras of  $gl_{n}$ and $sl_{n}$  with respect to their Gelfand-Tsetlin subalgebras. 

We will need the following action of $\mathcal M$ on $L$.     Suppose   $\mathcal M = \mathbb Z^{N}$  is a free abelian group of rank $N$ generated by $\sigma_{i}, i=1, \ldots, N$.      We say that   $\mathbb Z^{N}$ acts by shifts on $L$ if   $\sigma_i$, $i=1, \ldots, N$ act on $t_j$ as follows:  $\sigma_i(t_j)=t_j-\delta_{ij}$. 
Thus we can construct  a skew group algebra $L*\mathcal M$.  
As a subgroup of automorphisms of $L$,  $G$ normalizes $\mathcal M$
$$\sigma_{k}^{\pi}=\delta_{\pi(k)}, \pi\in S_{p}, \sigma_{k}^{\epsilon_{i}}=\left\{\begin{array}
{ll}\sigma_{k},&\text{ if } i\neq k\\
\sigma_{k}^{-1}&\text{ otherwise, }
\end{array}\right.$$

Hence $G$ acts on $L*\mathcal M$.  We will call this action \emph{natural}.


Let $R_N=\c[t_1, \ldots, t_N]*\mathbb Z^N$, where the group $\mathbb Z^N$ is generated by the elements  $\sigma_i$, $i=1, \ldots, N$ as above. For each $i=1, \ldots, N$ and any $c\in k$ consider the  involutions $\epsilon_{R_N,c,i}^{\pm}$ on $R_N$ defined by $\epsilon_{R_N,c,i}^{\pm}(\sigma_i)=\pm\,\sigma_i^{-1}$,    $\epsilon_{R_N,c,i}^{\pm}(\sigma_j)=\sigma_j$ if $i\neq j$,    $\epsilon_{R_N,c,i}^{\pm}(t_i)=c-t_i$ and 
 $\epsilon_{R_N,c,i}^{\pm}(t_j)=t_j$ if $j\neq i$.

\begin{lemma}
\label{lemma-on-transformation-inversion-mult-inversion-add} Let $\Tilde{A_{N}}$ be the localisation of the Weyl algebra $A_N$ by the multiplicative set generaled by $\{x_1, \ldots, x_N\}$.
The homomorphism $\phi^{\pm}_{c}:\Tilde{A_{N}}\rightarrow R_N$, given by $$\phi^{\pm}_{c}(x_i)=\sigma_i,\ \phi^{\pm}_{c}(\partial_i)=\big(t_i+1-\dfrac{c}{2}\big)\sigma_i^{-1}+ (1\mp\sigma_i^{-2}),$$ $i=1, \ldots, N$ is an isomorphism of algebras with involutions.
\end{lemma}

\begin{proof} It is sufficient to consider the case $N=1$. Set $R=R_1$.
Since $\phi^{\pm}_{c}(\partial x - x \partial)=t-\sigma t\sigma^{-1}=1$, $\phi^{\pm}_{c}$ are homomorphisms, $(\phi^{\pm}_{c})^{-1}(\sigma)=x$ and $(\phi^{\pm}_{c})^{-1}(t)=\big(\partial x+\dfrac{1}{2}\big)+(\dfrac{1}{x}\mp x).$
We also have $$\epsilon_{R,c}^{\pm}\phi^{\pm}_{c}(\partial)=\mp\big(t-1-\dfrac{c}{2}\big)\sigma+ (1\mp\sigma^{2})$$ and
$$ \phi^{\pm}_{c}\epsilon_{\tilde{A}_{1},c}^{\pm}(\partial)= \mp\sigma^{2}
\bigg(\big(t+1-\dfrac{c}{2}\big)\sigma^{-1}+ (1\mp\sigma^{-2}) \bigg)=\mp\big(t-1-\dfrac{c}{2}\big)\sigma+ (1\mp\sigma^{2}).$$
\end{proof}

For integers $n\geq 1$, $m\geq 0$ denote $A_{n,m}$ the $n$-th Weyl algebra over the field of rational functions $\c(z_1, \ldots, z_m)$. 
Then $A_{n,m}$ admits the skew field of fractions  $F_{n, m}\simeq F_n\otimes \c(z_1, \ldots, z_m)$.  We have

\begin{lemma}
\label{lem-skew-main}
Let $\mathcal K=(L*\mathcal M)^G$ be a linear Galois algebra where $G=G_N$ is a classical Weyl group and
\begin{itemize}
\item
 $L=\c(t_{1},\dots, t_{N})$; 
 \item $G$ acts naturally on $\mathcal K$;
\item
$\mathcal M\simeq \mathbb Z^{n}$  acts by shifts on  $t_{1},\dots, t_{n}$, $n\leq N$.
\end{itemize}  Then $\mathcal K$ admits the skew field of fractions $F(\mathcal K)$ and 
$F(\mathcal K)\simeq F_{n,N-n}.$

\end{lemma}

\begin{proof}
 Since the action of $\mathcal M$ is trivial on $L(t_{n+1}, \ldots, t_N)$ then we have the   $G$-equivariant embedding $$(L(t_1, \ldots, t_n)*\mathcal M)\otimes L(t_{n+1}, \ldots, t_N) \hookrightarrow L*\mathcal M$$ and 
 $$((L(t_1, \ldots, t_n)\otimes L(t_{n+1}, \ldots, t_N))*\mathcal M)^G\hookrightarrow (L*\mathcal M)^G.$$  Moreover, both algebras have the same skew fields.  But $$((L(t_1, \ldots, t_n)*\mathcal M)\otimes L(t_{n+1}, \ldots, t_N))^G\simeq ((L(t_1, \ldots, t_n)*\mathcal M)^G\otimes L(t_{n+1}, \ldots, t_N)^G.$$ Since $(L(t_1, \ldots, t_n)*\mathcal M\simeq R_n$, $R_n^G=R_n^{G_n}$, $F(R_n)\simeq F(A_n)$ and $L(t_{n+1}, \ldots, t_N)^G\simeq 
 L(z_1, \ldots z_{N-n})$ we obtain  $$F(\mathcal K)\simeq F(A_n)^G\otimes L(z_1, \ldots z_{N-n}).$$ The result follows from Theorem~\ref{main}. 

\end{proof}

\begin{theorem}
\label{theorem-main-theorem}
Let $U$ be a linear Galois algebra  in $(L*\mathcal M)^G$ such that
\begin{itemize}
\item
 $L=\c(t_{ij}, i=1, \ldots, N; j=1, \ldots, n_i; z_1, \ldots z_m)$, for some integers $m$, $n_1$, $\ldots$, $n_N$; 
 \item $G=G_1\times \ldots \times G_N$, where $G_s$ acts normally only on  variables $t_{s1}$, $\ldots$, $t_{s, n_s}$, $s=1, \ldots, N$;
\item
$\mathcal M\simeq \mathbb Z^{n}$  acts by shifts on  $t_{11},\dots, t_{N,n_N}$, $n=n_1+\ldots + n_N$.
\end{itemize}  Then $U$ admits the skew field of fractions $F(U)$ and 
$F(U)\simeq F_{n, m}.$
\end{theorem}

\begin{proof}
Follows from Theorem 3 and Lemma \ref{lem-skew-main}.

\end{proof}

\begin{remark}
\label{remark-case-of-s-n} The  action of the symmetric group $S_{N}$  on $\c^{N}$ by permutations of the coordinates is obviously   linear and it normalizes the action of $\mathcal M=\mathbb Z^{N}$ on $\c^{N}$ by shifts. Recall that $U(gl_{n})$ and $U(sl_{n})$ are linear Galois algebras with respect to their Gelfand-Tsetlin subalgebras \cite{FO}. Then Theorem~\ref{theorem-main-theorem}  implies immediately the Gelfand-Kirillov conjecture for $gl_{n}$ and $sl_{n}$. In a similar manner one obtains the Gelfand-Kirillov conjecture for restricted Yangians of type $A$ which was shown in \cite{FMO}.
\end{remark}

\end{document}